\documentclass{article}
\usepackage{amsmath,amsthm,amsfonts,amscd,amssymb,bm}
\usepackage{mathrsfs,listings,url}
\usepackage{graphics,graphicx,color}
\usepackage{float}
\allowdisplaybreaks
\usepackage[titletoc,page]{appendix}
\usepackage{subfigure}
\usepackage{epigraph}
\usepackage[colorlinks=true, bookmarksopen,
            pdfauthor={CST},
            pdfcreator={pdftex},
            pdfsubject={algorithms},
            linkcolor={blue},
            anchorcolor={black},
            citecolor={firebrick},
            filecolor={magenta},
            menucolor={black},
            pagecolor={red},
            plainpages=false,pdfpagelabels,
            urlcolor={db} ]{hyperref}
\usepackage{comment}
\usepackage{verbatim}
\usepackage{marginnote}
\usepackage{float}
\usepackage[makeroom]{cancel}
\usepackage[section]{placeins}
\usepackage{mathptmx}
\usepackage{cases}
\usepackage{subeqnarray}
\usepackage{hhline}
\usepackage[linesnumbered,ruled,vlined]{algorithm2e}
\usepackage{xcolor}
\usepackage{tikz}
\usetikzlibrary{arrows}
\usetikzlibrary{positioning,shapes,snakes,calc,decorations,decorations.markings}

\newtheorem{remark}{Remark}[section]

\newtheorem{theorem}{Theorem}[section]
\newtheorem{proposition}{Proposition}[section]

\includecomment{confidential}
\excludecomment{confidential}

\restylefloat{table}
\allowdisplaybreaks
\definecolor{db}{rgb}{0.0470,0,0.5294}
\definecolor{dg}{rgb}{0.0,0.392,0.0}
\definecolor{firebrick}{rgb}{0.698,0.133,0.133}
\definecolor{bl}{rgb}{0.0,0.0,0.0}
\definecolor{linen}{rgb}{0.980,0.941,0.902}
\definecolor{ivory}{rgb}{1.0,1.0,0.941}
\definecolor{aliceblue}{rgb}{0.941,0.973,1.0}
\definecolor{beige}{rgb}{0.961,0.961,0.863}
\definecolor{tan}{rgb}{0.824,0.706,0.549}
\definecolor{lightsteelblue}{rgb}{0.690,0.769,0.871}
\definecolor{paleturquoise}{rgb}{0.686,0.933,0.933}
\definecolor{lightblue}{rgb}{0.678,0.847,0.902}
\definecolor{skyblue}{rgb}{0.529,0.808,0.922}
\definecolor{palegoldenrod}{rgb}{0.933,0.910,0.667}
\definecolor{lightgoldenrod}{rgb}{0.933,0.867,0.510}
\definecolor{lightyellow}{rgb}{1.0,1.0,0.878}
\definecolor{yellow}{rgb}{1.0,1.0,0.0}
\definecolor{lightyellow1}{rgb}{1.0,1.0,0.878}
\definecolor{lemonchiffon}{rgb}{1.0,0.980,0.804}
\definecolor{myyellow}{rgb}{1,1,.9}
\definecolor{darkgreen}{rgb}{0.0,0.392,0.0}
\definecolor{darkviolet}{rgb}{0.580,0.0,0.827}
\definecolor{lightsalmon}{rgb}{1.0,0.627,0.478}
\definecolor{orange}{rgb}{1.0,0.647,0.0}
\definecolor{darkblue}{rgb}{0.00,0.00,0.55}

\numberwithin{equation}{section}

\begin{document}


\title{
Refactorization of a variable step, unconditionally stable method of Dahlquist, Liniger and Nevanlinna}

%

\author{William Layton\thanks{%
Department of Mathematics, University of Pittsburgh, Pittsburgh, PA 15260,
USA. Email: \href{mailto:wjl@pitt.edu}{wjl@pitt.edu}. The research herein was partially supported by NSF grants DMS 1817542 and 2110379.}
\and Wenlong Pei\thanks{%
Department of Mathematics, University of Pittsburgh, Pittsburgh, PA 15260,
USA. Email: \href{mailto:wep17@pitt.edu}{wep17@pitt.edu}. The research herein was partially supported by NSF grants DMS 1817542 and 2110379.}
\and Catalin Trenchea\thanks{%
Department of Mathematics, University of Pittsburgh, Pittsburgh, PA 15260,
USA. Email: \href{mailto:trenchea@pitt.edu}{trenchea@pitt.edu}.}}

\maketitle


\begin{abstract}
The one-leg, two-step time-stepping scheme proposed by Dahlquist, Liniger and Nevanlinna
has clear advantages in complex, stiff numerical simulations: unconditional $G$-stability for variable time-steps and second-order accuracy.
Yet it has been underutilized due, partially, to its complexity of direct implementation.
We prove herein that this method is equivalent to the backward Euler method with pre- and post arithmetic steps added. This refactorization eases implementation in complex, possibly legacy codes.
The realization we develop
reduces complexity, including cognitive complexity
and increases accuracy over
both first order methods and constant time steps second order methods.
\end{abstract}

\section{Introduction}
Numerical methods for evolution equations are designed based on accuracy and stability.
The theory of both is highly developed for constant time step and linear problems.
Less is known for
variable time steps
and nonlinear problems.
These cases are subtle.
For example, for increasing time steps, the BDF2 method loses A-stability and suffers non-physical energy growth in the approximate solution \cite{MR723632}.
Even the trapezoidal
method is unstable when used with variable time steps, see e.g. \cite{MR714701}, \cite[pp. 181-182]{MR0426438}.
Dahlquist, Liniger and Nevanlinna in  \cite{MR714701} proposed a one parameter $\delta$-family of variable-step, one-leg, two-step methods \eqref{eq:1-legDLN},
which are second-order accurate, and variable-step, nonlinearly, long-time stable.
Its detailed specification (given in Section \ref{subsec:DLNandTimeFilter}),
is sufficiently Gordian to deter its
use  in complex applications, in which a method with DLN's excellent
properties should be valued.
Our preliminary work on adaptive time-stepping for flow problems \cite{LPQT21,MR4236432} shows that \eqref{eq:1-legDLN} has promise, motivating the work herein.

{\it Refactorization} generally means restructuring of an existing
algorithm
without changing its behaviour.
The goal of refactorization is to reduce complexity by creating a
simple and clean
logical structure,
improving implementation, code readability, source maintainability and extensibility.
\normalcolor
Herein we show how \eqref{eq:1-legDLN} can be
refactorized to be
easily implemented in an intricate, possibly legacy/black-box code, without modifying the `assemble and solve' portion.
While our refactorization can work for other base methods, to fix ideas
for $y' = f \left(t,y \right)$,
we consider a method based on the fully implicit Euler method
\begin{align}
\label{eq:fulImmethod}
\tag{BE}
\frac{y^{\rm new} - y^{\rm old}}{t^{\rm new} - t^{\rm old}} = f \left(t^{\rm new} , y^{\rm new} \right).
\end{align}
Figure \ref{tikz-DLN} illustrates the implementation of the
\eqref{eq:1-legDLN} method
in Algorithm \ref{alg0},
by adding a pre-filter step to the data ahead of the nonlinear solver \eqref{eq:fulImmethod}, and a poster-filter step after the solver \eqref{eq:fulImmethod}.
%
This algorithmic idea is our 
first contribution. In Section \ref{sec:variDLN} we prove a new expression for the local truncation error, which simplifies the time-adaptive implementation of \eqref{eq:1-legDLN}, and also recall its variable step $G$-stability property.
%
%
%
%
\section{The DLN method and its refactorization}
\label{subsec:DLNandTimeFilter}
\noindent
We consider a numerical approximation of
the initial value problem
\begin{align}
\label{eq:IVP}
y'(t) = f\left(t,y(t) \right),
\qquad
y(0) = y_{0}.
\end{align}
at times $\{t_{n} \}_{n \geq 0}$, with the time step $k_{n} = t_{n+1} - t_{n}$.
To present the method of \cite{MR714701},
let $\varepsilon_{n} = (k_{n} - k_{n-1}) / (k_{n}+k_{n-1})$ 
denote the step size variability and
$\delta\in [0,1]$ be an arbitrary parameter.
The \eqref{eq:1-legDLN} method is a 2-step method with
coefficients:
\begin{align}
\label{eq:alpha-beta}
\begin{bmatrix} \alpha_{2} \\ \ \\ \alpha_{1} \\ \ \\ \alpha_{0} \end{bmatrix} = \begin{bmatrix} \frac{1}{2}(\delta+1) \\ \ \\ -\delta \\ \ \\ \frac{1}{2}(\delta-1) \end{bmatrix}, \ \ \ \ \ \ \ \ \
\begin{bmatrix} \beta_{2}^{(n)} \\ \ \\ \beta_{1}^{(n)} \\ \ \\ \beta_{0}^{(n)} \end{bmatrix} =
\begin{bmatrix} \frac{1}{4}\left(1+ \frac{1-\delta^{2}}{(1+\varepsilon_{n} \delta)^{2}} + \varepsilon_{n}^{2} \frac{\delta(1-\delta^{2})}{(1+\varepsilon_{n} \delta)^{2}}+ \delta \right) \\ \
\\ \frac{1}{2}\left(1 - \frac{1-\delta^{2}}{(1+\varepsilon_{n} \delta)^{2}} \right) \\ \ \\
\frac{1}{4}\left(1+ \frac{1-\delta^{2}}{(1+\varepsilon_{n} \delta)^{2}} - \varepsilon_{n}^{2} \frac{\delta(1-\delta^{2})}{(1+\varepsilon_{n} \delta)^{2}}- \delta \right) \end{bmatrix}.
\end{align}
Note that $\{\alpha_{\ell}\}_{\ell=0}^{2}$ are step size independent, while $\{\beta_{\ell}^{(n)}\}_{\ell=0}^{2}$ are step size dependent.
Define the average step size $\widehat{k}_{n}$ as follows:
\begin{align}
\label{eq:stephat}
\widehat{k}_{n} = \alpha_{2} k_{n} - \alpha_{0} k_{n-1}
= \delta \frac{k_{n} - k_{n-1}}{2} + \frac{k_{n}+k_{n-1}}{2}
.
\end{align}
The variable step DLN method of \cite{MR714701}
as a one-leg
\footnotemark
\footnotetext{
The `one-leg' term was coined by Dahlquist in 1975
\cite{MR0448898} to name the
multistep methods which involve only one value of $f$ in each step.
In particular, the leapfrog and BDF methods are
one-leg multistep methods.
}
method is
\begin{align}
\label{eq:1-legDLN}
\tag{DLN}
\frac{\alpha_{2}y_{n\!+\!1} \!+\! \alpha_{1}y_{n} \!+\! \alpha_{0}y_{n\!-\!1} }{\widehat{k}_{n}}
\!=\!
f \!\Big(\! \beta_{2}^{(n)}t_{n\!+\!1} \!+\! \beta_{1}^{(n)}t_{n} \!+\! \beta_{0}^{(n)}t_{n\!-\!1}, \beta_{2}^{(n)}y_{n\!+\!1} \!+\!  \beta_{1}^{(n)}y_{n} \!+\! \beta_{0}^{(n)}y_{n\!-\!1} \!\!\Big)\!.
\end{align}
\begin{remark}
The \eqref{eq:1-legDLN} methods are indexed by the free parameter $\delta\in [0,1]$.
When $\delta = 1$, the \eqref{eq:1-legDLN} method becomes
the (implicit) midpoint rule \cite{MR4092601,MR4265875}
\begin{align}
\label{one-step midpoint}
\tag{one-step midpoint}
\frac{y_{n+1}-y_n}{k_n}
=
f\Big( \frac{1}{2}(t_{n+1}+t_n) ,  \frac{1}{2}( y_{n+1} + y_n) \Big),
\end{align}
while for $\delta=0$, the \eqref{eq:1-legDLN} method is the
(implicit) midpoint rule
with double time step
\begin{align}
\label{two-step midpoint}
\tag{two-step midpoint}
\frac{y_{n+1}-y_{n-1}}{k_n+k_{n-1}} = f\Big( \frac{1}{2}(t_{n+1}+t_{n-1}) ,  \frac{1}{2}( y_{n+1} + y_{n-1}) \Big).
\end{align}
\end{remark}

To reduce the complexity of implementing \eqref{eq:1-legDLN}, we consider its implementation 
through
 pre- and post-processes of an implicit (backward) Euler  method,
 described in Algorithm \ref{alg0}, and illustrated in
 Figure \ref{tikz-DLN}.
\\
\LinesNumberedHidden
\begin{algorithm}[H]
\label{alg0}
    \caption{
    Refactorization of the \eqref{eq:1-legDLN} method
    }
    \label{alg:DLNfilter1step}
    \KwIn{$y_{n}$, $y_{n-1}$ and $t_{n-1},t_n,t_{n+1}$ \;}
    \tcp{Pre-process :~\color{blue}interpolation\normalcolor}
    \tcp{Evaluate quantities in \eqref{eq:alpha-beta} and \eqref{eq:stephat}
    }
%
\centerline{$
\qquad
\begin{array}{l}
 \alpha_{2} = \frac{1}{2}(\delta+1), \quad
  \alpha_{1} = - \delta, \quad
  \alpha_{0} = \frac{1}{2}(\delta-1) ,\quad
\\
\beta_{2}^{(n)} = \frac{1}{4}\left(1+ \frac{1-\delta^{2}}{(1+\varepsilon_{n} \delta)^{2}} + \varepsilon_{n}^{2} \frac{\delta(1-\delta^{2})}{(1+\varepsilon_{n} \delta)^{2}}+ \delta \right) , \quad
\\
\beta_{1}^{(n)} = \frac{1}{2}\left(1 - \frac{1-\delta^{2}}{(1+\varepsilon_{n} \delta)^{2}} \right), \quad
\beta_{0}^{(n)} =
 1 - \beta_{2}^{(n)}  - \beta_{1}^{(n)},
 \quad
\widehat{k}_{n} = \alpha_{2} k_{n} - \alpha_{0} k_{n-1}.
\end{array}$}
\hspace{-.6cm}
  \tcp{ Define the refactorization coefficients}
  \vspace{-.5cm}
\begin{align}
\label{eq:coeffDLNfilterVari}
\left\{
\begin{array}{l}\displaystyle
a_{1}^{(n)} = \beta_{1}^{(n)} - \frac{\alpha_{1} \beta_{2}^{(n)} }{\alpha_{2}},
\quad
a_{0}^{(n)} =
1- a_{1}^{(n)} ,
\quad
b^{(n)} = \frac{\beta_{2}^{(n)}}{\alpha_{2}} ,
\quad
\\
c_{2}^{(n)} = \frac{1}{\beta_{2}^{(n)}}
,
\quad
c_{1}^{(n)} = -\frac{\beta_{1}^{(n)}}{\beta_{2}^{(n)}},
\quad
c_{0}^{(n)} = -\frac{\beta_{0}^{(n)}}{\beta_{2}^{(n)}}.
\end{array}
\right.
\end{align}
\vspace{-.3cm}

    \tcp{Evaluate the time-step for BE}
    $(\Delta t)_{n}^{\text{BE}} \Leftarrow b^{(n)} \widehat{k}_{n}$

    \tcp{Set the BE time interval: $[t_{}^{\text{new}} - (\Delta t)_{n}^{\text{BE}}, t_{}^{\text{new}}]$, and $y_{n}^{\text{old}}$}
    $t_{}^{\text{new}} \Leftarrow \beta_{2}^{(n)} t_{n+1} + \beta_{1}^{(n)}t_{n} + \beta_{0}^{(n)}t_{n-1}$;
    \quad
    $y_{}^{\text{old}} \Leftarrow a_{1}^{(n)} y_{n} + a_{0}^{(n)} y_{n-1}$ \;
    \:
    \tcp{\color{blue} backward Euler}
\noindent\fbox{%
    \parbox{11.0cm}{%
    Solve for $y_{}^{\text{new}}$:
    \qquad
    $\displaystyle\frac{y_{}^{\text{new}} - {y}_{}^{\text{old}}}{(\Delta t)_{n}^{\text{BE}}} = f \left(t_{}^{\text{new}},y_{}^{\text{new}}  \right)$

    }%
}

    \tcp{Post-process :~\color{blue}extrapolation\normalcolor}
    $y_{n+1} \Leftarrow  c_{2}^{(n)}y_{}^{\text{new}} + c_{1}^{(n)}y_{n} + c_{0}^{(n)}y_{n-1}  $    \tcp*{the DLN solution}
    \KwOut{$y_{n+1}
    		$, \quad If desired: Estimate Error and adapt $k_n$}
%
    ${\empty}$
\end{algorithm}
%
\noindent
%
Since
$
\alpha_0 + \alpha_1 + \alpha_2= 0,
\beta_0^{(n)} + \beta_1^{(n)} + \beta_2^{(n)} = 1,
$
the coefficients $a_i^{(n)}, b^{(n)}, c_i^{(n)}$ satisfy 
$
a_0^{(n)}+a_1^{(n)}
= 1,
c_{2}^{(n)} + c_{1}^{(n)}  + c_{0}^{(n)}
=1.
$
\normalcolor
%
%
%
\begin{figure}
\resizebox{.95\textwidth}{!}{
\begin{tikzpicture}
  [
    ->,
    >=stealth',
    auto,node distance=3cm,
    thick,
    main node/.style={ draw, black,font=\sffamily\Large\bfseries}
    ]

  \node[main node] (1)   {$t_{n-1}$};
  \node[main node] (2) [darkgreen,ellipse] [right of=1] {$t^{\rm old}$};
  \node[main node] (3) [right of=2] {$t_{n}$} ;
  \node[main node] (4) [darkgreen,ellipse,right of=3] {$t^{\rm new}$};
  \node[main node] (5) [right of=4] {$t_{n+1}$};
%
  \node[main node] (11) [above=3] {$y_{n-1}$};
    \node[main node] (12) [right of=11,darkgreen,ellipse,above=0.5] {$y^{\rm old}$};
   \node[main node] (13) [right of=12, above=0.40] {$y_{n}$};
   \node[main node] (14) [darkgreen,ellipse,right of =13, below=0.5] {$y^{\rm new}$} ;
   \node[main node] (15) [right of=14, above=0.15] {$y_{n+1}$};

  \path[every node/.style={font=\sffamily\small}]
    (1) edge node [right] {} (2)
    (2) edge node [right] {} (3)
    (3) edge node [right] {} (4)
    (4) edge[line width=.1mm] node [right] {} (5);

 \draw[cyan, dotted, thick]
    (1) [out=30, in=155] to node[midway,below] {pre-process} (4) ;
    \draw[cyan, dotted, thick]
    (3)  [out=90, in=85] to  node[midway,below] {pre-process} (4);
    \draw[cyan, dotted, thick]
     (5) [out=90, in=90] to node[above,midway] {pre-process} (4)  ;
  \draw[darkgreen, dotted, thick]
    (2) [out=-30, in=-150] to  node[midway,below] {$(\Delta t)_n^{\rm BE}$} (4);

\draw[firebrick] (11) to[out=-5,in=205] (13) to[out=25,in=90] node[midway,above] {DLN} (15); 

   \draw[cyan, dotted, thick]
    (11) [out=90, in=150] to  node[midway,below] {pre-process} (12);
   \draw[cyan, dotted, thick]
    (13) [out=160, in=65] to node[midway,above] {pre-process} (12);
   \draw[darkgreen,  very thick]
    (12) [out=-40, in=-150] to node[midway,below] {Backward Euler} (14);
   \draw[cyan, dashed, thick]
    (14) [out=40, in=90] to  node[midway,below] {post-process} (15);

  \path
    (1) edge node [right] {} (2)
    (2) edge node [right] {} (3)
    (3) edge node [right] {} (4)
    (4) edge[line width=.25mm] node [right] {} (5);
\end{tikzpicture}
}
\caption{\textbf{Refactorization of the \eqref{eq:1-legDLN} method
 as a pre- and post-processed \eqref{eq:fulImmethod} method}}
\label{tikz-DLN}
\end{figure}
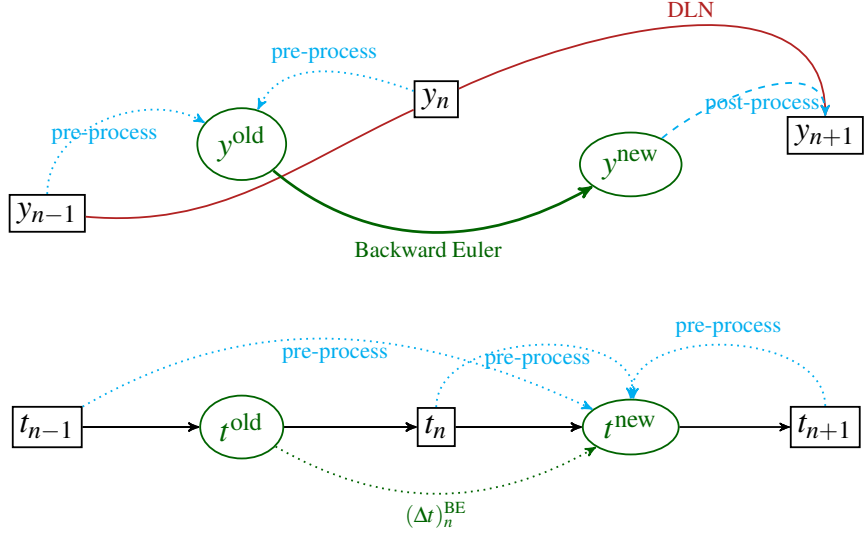

\begin{theorem}
\label{th:DLN2BE}
Algorithm \eqref{alg:DLNfilter1step} is equivalent to the \eqref{eq:1-legDLN} method.
\end{theorem}
\begin{proof}
First using the notations \eqref{eq:coeffDLNfilterVari},
the post-processing step writes
\begin{align*}
y^{\rm new} = \frac{1}{c_2^{(n)}}y_{n+1} - \frac{c_1^{(n)}}{c_2^{(n)}}y_{n} - \frac{c_0^{(n)}}{c_2^{(n)}}y_{n-1}
=  \beta_2^{(n)} y_{n+1} + \beta_1^{(n)} y_{n}  + \beta_0^{(n)} y_{n-1}.
\end{align*}
Using also the pre-processing relations, the backward Euler step in Algorithm \ref{alg:DLNfilter1step}
\begin{align}
\label{eq:DLN2BE}
\tag{DLN2BE}
    \frac{y_{}^{\text{new}} - y_{}^{\text{old}}}{ (\Delta t)_{n}^{\text{BE}} }
= f \left(t_{}^{\text{new}},y_{}^{\text{new}}  \right)
\end{align}
translates to
\begin{align*}
&
\frac{1}{\widehat{k}_{n}} \bigg(
 \frac{1}{b^{(n)}c_2^{(n)}} y_{n+1} - \frac{1}{b^{(n)}}\Big( \frac{c_1^{(n)}}{c_2^{(n)}} + a_1^{(n)}\Big) y_n - \frac{1}{b^{(n)}}\Big( \frac{c_0^{(n)}}{c_2^{(n)}} + a_0^{(n)}\Big) y_{n-1}
 \bigg)
\\
&
=
f \left( \beta_2^{(n)} t_{n+1} + \beta_1^{(n)} t_{n}  + \beta_0^{(n)} t_{n-1} , \beta_2^{(n)} y_{n+1} + \beta_1^{(n)} y_{n}  + \beta_0^{(n)} y_{n-1}  \right).
\end{align*}
Finally, by \eqref{eq:coeffDLNfilterVari}, this shows that the backward-Euler based
Algorithm \ref{alg:DLNfilter1step}
yields
the solution of the \eqref{eq:1-legDLN} method.
\end{proof}

\subsection{Related Work}
The \eqref{eq:1-legDLN} method is variable-step $G$-stable outgrowth of a method of Liniger \cite{MR723840}, which is non-autonomous $A$-stable (i.e. for $y'=\lambda (t) y$).
The pre- and post-process steps in the Algorithm \ref{alg0}
are akin to time filters,
highly
developed as numerical methods in atmospheric science
\cite{Asselin_72,Robert69,williams2013achieving,YLCT14,MR3807176}.
Recently it was noticed in \cite{MR3803857} that this technique for adding stability can also increase accuracy.
The idea of prefilter $\rightarrow$ simple method $\rightarrow$ postfilter was develloped in a different direction for constant time steps in \cite{decaria2020general}.

The refactorization
of an algorithm to reduce its cognitive complexity
has been used in \cite{MR4092601} to rearrange a family of one-leg one-step methods into a  backward Euler code followed by post-processing,
and further applied for partitioning multi-physics problems \cite{MR4230415,MR4176075,MR4265875,MartinaBukacCatalinTrenchea2021}.
In \cite{MR740863},
the authors
describe the implementation of the \eqref{eq:1-legDLN} formulas in a Nordsieck formulation
\cite{MR136519,MR537967} essentially identical to that of the backward differentiation formulas, facilitating to adapt Nordsieck formulation codes like DIFSUB \cite{MR0315898,6f8955503265487284e8c15e900eb28a} to the \eqref{eq:1-legDLN} formulas.
%
%

\section{Convergence analysis of \eqref{eq:1-legDLN}}
\label{sec:variDLN}
%
While stability and consistency were already addressed in \cite{MR714701}, we
present
complementary
details on both which are useful for developing an adaptive \eqref{eq:1-legDLN} method.

\subsection{Consistency error}%
\label{section_consistency-error}
%
In \cite{MR2139780,MR2284691}, the variable time-step \eqref{eq:1-legDLN} method was implemented in an adaptive manner, using a local and global error estimator.
Similar to \cite{MR714701}, the authors of \cite{
MR740863,MR2139780,MR2284691,MR0448898} use
the classical definition
of the
local truncation error
\begin{align*}
{\cal L}_1
 \big( y(t),t_{n+1},k_{n} \big)
=
&
\frac{1}{{\widehat{k}_{n}}}\sum_{\ell=0}^{2} \alpha_{\ell}y(t_{n-1 + \ell})
-f \left( t_{n,\beta},
\beta_{2}^{(n)} y(t_{n+1}) +  \beta_{1}^{(n)} y(t_{n}) + \beta_{0}^{(n)} y(t_{n-1})
\right),
\end{align*}
where 
$
	t_{n,\beta} = \beta_{2}^{(n)}t_{n+1} + \beta_{1}^{(n)}t_{n} + \beta_{0}^{(n)}t_{n-1}
$. 
 The
 above definition
 follows the approach taken in the analysis of linear multistep methods (see e.g. \cite[page 27]{MR0423815}), and involves both
 the differentiation defect
 ${\cal L}_{d} =
	\frac{1}{\widehat{k}_{n}}
	\sum_{\ell=0}^{2} \alpha_{\ell}y(t_{n-1 + \ell})-f \big( t_{n,\beta}, y(t_{n,\beta})
$,
 and the interpolation defect
${\cal L}_{i} =
	f \big( t_{n,\beta}, y(t_{n,\beta}) \big)
	- f \big( t_{n,\beta}, \sum_{\ell=0}^{2} \beta_{\ell}^{(n)} y(t_{n-1+\ell}) \big)
$.
Dahlquist
 raised in \cite{MR723829} the question of the appropriateness of this viewpoint:
``We accept this definition, but {\it we do not accept
${\cal L}_{1}$
as the adequate local truncation error!"}
%

Using the refactorized form \eqref{eq:DLN2BE} 
and Theorem \ref{th:DLN2BE},
we now prove
that the local truncation error of the one-leg \eqref{eq:1-legDLN} method can be evaluated only by the differentiation defect \eqref{eq:DLN_defect}, similarly to
the midpoint rule \cite{MR4092601} and the Runge-Kutta methods.
The new expression \eqref{eq:DLN_defect} simplifies greatly the
error estimation.
\normalcolor
\normalcolor

\begin{confidential}
\color{firebrick}
\\

Wenlong: can you look at Theorem on page 1136 on \cite{MR0448898}, its proof, and the remarks/discussion following the Theorem, and what is the relationship with our Proposition below???
There is statement there saying `{\it If \ldots, $c = 0$ and $p_d = k$, the differentiation error is the dominant part of the error.}'
\\

\normalcolor
\end{confidential}
\begin{proposition}
\label{prop:2ndVariDLN}
The local truncation error of \eqref{eq:1-legDLN} is
the differentiation error and
\begin{align}
\tag{LTE}
\label{eq:DLN_defect}
{\cal L}_d \big( y(t),t_{n+1},k_{n} \big)
\approx
{\frac{y'''(t_{n})}{2}} \Big[ \frac{1}{3 \widehat{k}_{n}} \big({k_{n}^{3}} - \frac{\alpha_{0}}{\alpha_{2}}{k_{n-1}^{3}}  \big) - \frac{1}{\alpha_{2}} \big(\beta_{2}^{(n)}k_{n} - \beta_{0}^{(n)}k_{n-1} \big)^{2} \Big]
.
\notag
\end{align}
\end{proposition}
\begin{proof}
The consistency order and the coefficient of the leading term in
\eqref{eq:DLN_defect} follow by
Taylor expansions.
%
On one hand,
since \eqref{eq:1-legDLN} can be refactorized as the one-step method \eqref{eq:DLN2BE},
we further write the \eqref{eq:1-legDLN} method as follows
\begin{align}
\label{eq:DLNquadrature}
\frac{\alpha_{2}y_{n+1} + \alpha_{1}y_{n} + \alpha_{0}y_{n-1} }{\widehat{k}_{n}}
= f \left(t_{}^{\text{new}},y_{}^{\text{new}}  \right).
\end{align}
On the other hand,
when we integrate \eqref{eq:IVP} on $[t_{n-1}, t_n]$ and on $[t_n,t_{n+1}]$,
\begin{confidential}
{\color{cyan}
\begin{align*}
&
\tag{$|\cdot \frac{1-\delta}{2}$}
y(t_n) - y(t_{n-1}) = \int_{t_{n-1}}^{t_n} f(t,y(t))\, dt
\\
&
\tag{$|\cdot \frac{1+\delta}{2}$}
y(t_{n+1}) - y(t_{n}) = \int_{t_{n}}^{t_{n+1}} f(t,y(t))\, dt
\end{align*}
}
\end{confidential}
multiply the results by $\frac{1-\delta}{2}$ and $\frac{1+\delta}{2}$, respectively, and add, we obtain
\begin{confidential}
{\color{cyan}
\begin{align*}
&
\frac{1+\delta}{2} y(t_{n+1})
- \frac{\cancel{1}+\delta}{2} y(t_{n})
+ \frac{\cancel{1}-\delta}{2} y(t_n) - \frac{1-\delta}{2} y(t_{n-1}) = \frac{1-\delta}{2} \int_{t_{n-1}}^{t_n} f(t,y(t))\, dt
+
\frac{1+\delta}{2} \int_{t_{n}}^{t_{n+1}} f(t,y(t))\, dt
\end{align*}
}
\end{confidential}
\begin{align*}
&
\mbox{$\frac{1+\delta}{2} y(t_{n+1})$}
- \delta y(t_{n})
-
\mbox{$\frac{1-\delta}{2}$}
 y(t_{n-1})
 =
 \mbox{$\frac{1-\delta}{2}$}
  \int_{t_{n-1}}^{t_n} f(t,y(t))\, dt
+
\mbox{$\frac{1+\delta}{2}$}
\int_{t_{n}}^{t_{n+1}} f(t,y(t))\, dt.
\end{align*}
Finally,
approximating both integrals in the right hand-side
with the chord quadrature rule,
with $t^\text{new}$ as the point of evaluation on both intervals, gives
\begin{confidential}
{\color{cyan}
\begin{align*}
\frac{1+\delta}{2} y(t_{n+1})
- \delta y(t_{n})
- \frac{1-\delta}{2} y(t_{n-1})
&
= \frac{1-\delta}{2} \int_{t_{n-1}}^{t_n} f(t,y(t))\, dt
+
\frac{1+\delta}{2} \int_{t_{n}}^{t_{n+1}} f(t,y(t))\, dt
\\
&
\approx \frac{1-\delta}{2} (t_{n} - {t_{n-1}} ) f(t^{\text{new}},y(t^{\text{new}}))
+
\frac{1+\delta}{2} ( t_{n+1} -  t_{n})  f(t^{\text{new}},y(t^{\text{new}}))
\\
&
=
\frac{1+\delta}{2} ( t_{n+1} -  t_{n})  f(t^{\text{new}},y(t^{\text{new}}))
+
\frac{1-\delta}{2} (t_{n} - {t_{n-1}} ) f(t^{\text{new}},y(t^{\text{new}}))
\\
&
=
\frac{1+\delta}{2} ( t_{n+1} -  t_{n})  f(t^{\text{new}},y(t^{\text{new}}))
-
\frac{\delta - 1 }{2} (t_{n} - {t_{n-1}} ) f(t^{\text{new}},y(t^{\text{new}}))
\\
&
=
\bigg( \frac{1+\delta}{2} ( t_{n+1} -  t_{n})
-
\frac{\delta - 1 }{2} (t_{n} - {t_{n-1}} )
\bigg) f(t^{\text{new}},y(t^{\text{new}}))
\end{align*}
}
\end{confidential}
\begin{align*}
&
\mbox{$\frac{1+\delta}{2} y(t_{n+1})$}
- \delta y(t_{n})
-
\mbox{$\frac{1-\delta}{2} y(t_{n-1})$}
\approx
\Big(
\mbox{$\frac{1+\delta}{2}$}
 ( t_{n+1} -  t_{n})
-
\mbox{$\frac{\delta - 1 }{2}$}
 (t_{n} - {t_{n-1}} )
\Big) f(t^{\text{new}},y(t^{\text{new}})) .
\end{align*}
which by \eqref{eq:alpha-beta}-\eqref{eq:stephat} yields the \eqref{eq:DLNquadrature} method.
\end{proof}
\begin{remark}
\label{remark-lte}
In particular, for $\delta = 1$ and $\delta = 0$,
from \eqref{eq:DLN_defect} we have that
\begin{align*}
{\cal L}^{\eqref{one-step midpoint}}
\approx
\mbox{$\frac{1}{24}$}
k_{n}^{3} y'''(t_n),
\qquad
{\cal L}^{\eqref{two-step midpoint}}
\approx
\mbox{$\frac{1}{24}$}
\left( k_{n} + k_{n-1}\right)^{3} y'''(t_n).
\end{align*}
\end{remark}

\begin{confidential}
\color{darkblue}
\begin{remark}
In \cite{MR723829}, Dahlquist suggests that one choice for the local truncation error of the general one-leg method could be
the sum of 	differentiation error and interpolation error, i.e.
\begin{align}
	{\cal L}_{d} &=
	\frac{1}{\widehat{k}_{n}}
	\sum_{\ell=0}^{2} \alpha_{\ell}y(t_{n-1 + \ell})-f \big( t_{n,\beta}, y(t_{n,\beta})
	\big), \tag{differential error} \\
	{\cal L}_{i} &=
	f \big( t_{n,\beta}, y(t_{n,\beta}) \big)
	- f \big( t_{n,\beta}, \sum_{\ell=0}^{2} \beta_{\ell}^{(n)} y(t_{n-1+\ell}) \big),
	\tag{interpolation error} \\
	{\cal L}_{1} &=
	{\cal L}_{d} + {\cal L}_{i}. \notag
\end{align}
However in the same paper, Dahlquist says that " We accept this definition, but we do not accept ${\cal L}_{1}$ as the adequate location error!" In our paper, we choose differential error as our definition of local truncation error for the one-leg DLN method in Proposition \ref{prop:2ndVariDLN}.
\end{remark}
\normalcolor
\end{confidential}
\subsection{G-stability}
Let $\langle \cdot , \cdot \rangle$ and $\Vert \cdot \Vert$ denote the inner product and $\ell^{2}$-norm in Euclidean space
$\mathbb{C}^{d}$.
For any pair of solutions $u(t), v(t)$ of \eqref{eq:IVP}, a
necessary and
sufficient condition \cite[page 384]{MR520750}
for $\|u(t)-v(t)\|$ to be a non-increasing function of $t$ is the {\it contractivity} (one-sided Lipschitz) condition on $f$:
\begin{align}
\label{contractivity}
\tag{contractivity}
{\rm Re} \big\langle f(t,u) - f(t,v) , u - v \big\rangle \leq 0,
\qquad \forall t\geq 0, \quad
\forall u,v\in {\mathbb C}^d.
\end{align}
The system \eqref{eq:IVP} for which $f$ satisfies the \eqref{contractivity} condition is 
called {\it dissipative}, see e.g. the Definition in \cite[page 268]{MR1127425}.
We recall that a Runge-Kutta method is B-stable, if the \eqref{contractivity} condition implies
$\|y_{n+1}-z_{n+1}\| \leq \|y_{n}-z_{n}\|$ for any $\{y_n\},\{z_n\}$ numerical solutions, see e.g. \cite[page 359]{Butcher1975}, or Definition 12.2 in \cite{MR2657217}.
Similarly, a
2-step linear multistep method
is called $G$-stable \cite{dahlquist_tritaNA76218,MR520750,MR519054,MR2657217} 
 if there exists a real positive definite matrix $G$ such that
its one-leg version is contractive,
namely $\|Y_{n+1} - Z_{n+1}\|_G \leq \|Y_{n} - Z_{n}\|_G$, where
$Y_n=[ y_{n}^{tr} ,  y_{n-1}^{tr}]^{tr}$.
%
In the case of the \eqref{eq:1-legDLN} method,  there exists such
a positive definite matrix (independent of the step size)
\begin{gather*}  
G(\delta) :=
\begin{bmatrix}
\frac{1}{4} (1+\delta) \mathbb{I}_d & 0 \vspace{.2cm} \\
0 & \frac{1}{4} (1-\delta) \mathbb{I}_d%
\end{bmatrix},
\qquad
\forall \delta\in [0,1].
\end{gather*}
As pointed out by 
Dahlquist in \cite{dahlquist_tritaNA7508},
both $B$-stability and $G$-stability imply $A$-stability, and $A$-stability implies $G$-stability for constant time steps.
\begin{proposition}
\label{thm:GstabDLNvari}
The
\eqref{eq:1-legDLN} method is unconditionally $G$-stable, and
\begin{align}
\label{eq:GstabDissipationVari}
\!\!\!
\Big\langle \sum_{\ell =0}^{2}{\alpha _{\ell }}y_{n-1+\ell } ,
	\sum_{\ell =0}^{2}{\beta _{\ell }^{(n)}}y_{n-1+\ell }\Big\rangle_{{\mathbb{R}}^{d}}
\!\!\!
=
\begin{Vmatrix}
{y_{n+1}} \\
{y_{n}}%
\end{Vmatrix}%
_{G(\delta )}^{2}
\!\!\!\!\!
-
\begin{Vmatrix}
{y_{n}} \\
{y_{n-1}}%
\end{Vmatrix}%
_{G(\delta )}^{2}
\!\!\!\!\!
+
\Big\| \sum_{\ell =0}^{2}{\gamma_{\ell }^{(n)}}y_{n-1+\ell } \Big\|^{2}
\!\!,
\end{align}
where the $\gamma$-coefficients
are
$
\gamma_{1}^{(n)}=-\frac{\sqrt{\delta ( 1-{\delta }^{2} ) }}{\sqrt{2}%
(1+\varepsilon _{n}\delta )},
\gamma_{2}^{(n)}=-\frac{1-\varepsilon _{n}}{2}%
\gamma_{1}^{(n)},
\gamma_{0}^{(n)}=-\frac{1+\varepsilon _{n}}{2}\gamma_{1}^{(n)}.
$
\end{proposition}
\noindent
The
`energy'
identity \eqref{eq:GstabDissipationVari}, implicit in \cite{MR714701},
follows from algebraic manipulations, see e.g. \cite{LPQT21}.
%
The $G$-stability of  \eqref{eq:1-legDLN},
i.e.
$\|Y_{n+1}-Z_{n+1}\|_{G(\delta)}
\leq \|Y_{n}-Z_{n}\|_{G(\delta)}
$,
follows from \eqref{eq:GstabDissipationVari}
and
the \eqref{contractivity} assumption.
%
The only  \eqref{eq:1-legDLN} methods which
yield the $\ell^2$ 
invariance of the solution
are the symplectic \eqref{one-step midpoint} and \eqref{two-step midpoint} rules:
%
the numerical dissipation 
$
\Big\| \sum_{\ell =0}^{2}{\gamma_{\ell }^{(n)}}y_{n-1+\ell } \Big\|
$
vanishes
if and only if
$\delta \in\{0, 1\}$.



\bibliographystyle{siam}

\end{document}